\documentclass[reqno,A4paper]{amsart}

\usepackage{amsmath}
\usepackage{amssymb}
\usepackage{amsthm}
\usepackage{enumerate}
\usepackage{mathrsfs} %%\mathcal classic,  \mathscr decorative
\usepackage{eqlist}
\usepackage{array}

\usepackage{hyperref}
\usepackage{graphicx}

\theoremstyle{plain}
\newtheorem{thm}{Theorem}[section]
\newtheorem{prop}[thm]{Proposition}
\newtheorem{lem}[thm]{Lemma}
\newtheorem{cor}[thm]{Corollary}
\theoremstyle{definition}

\newtheorem{rmk}[thm]{\textup{Remark}} %\textup for ``Remark'' is required
\newtheorem{examples}[thm]{\textit{Examples}} %\textit for ``Example'' is required

\newtheorem{question}[thm]{Question}

\newtheorem*{acknowledgement}{\textup{Acknowledgements}}
\numberwithin{equation}{section}
\newcommand{\ii}{\mathrm{i}}    %%complex unit

\newcommand{\NN}{\mathbb{N}}    %%natural numbers
\newcommand{\ZZ}{\mathbb{Z}}    %%integers
\newcommand{\QQ}{\mathbb{Q}}    %%rational numbers
\newcommand{\RR}{\mathbb{R}}    %%real numbers
\begin{document}

  \title[Further insights into the mysteries of the values of zeta functions at integers]{Further insights into the mysteries of the values of zeta functions at integers}

\author[J\'an Min\'a\v{c} \and Tung T. Nguyen \and Nguy$\tilde{\text{\^{e}}}$n Duy T\^{a}n]%
{ J\'an Min\'a\v{c}* \and Tung T. Nguyen** \and Nguy$\tilde{\text{\^{e}}}$n Duy T\^{a}n***}

\newcommand{\acr}{\newline\indent}

\address{\llap{*\,}Department of Mathematics\acr
                   Western University\acr
                    London, Ontario N6A 5B7\acr
                   CANADA}
\email{minac@uwo.ca}

\address{\llap{**\,}Brain and Mind Institute and Department of Mathematics\acr
                   Western University\acr
                    London, Ontario N6A 5B7\acr
                   CANADA}
\email{tungnt@uchicago.edu}

\address{\llap{***\,}School of Applied  Mathematics and 	Informatics\acr
                    Hanoi University of Science and Technology\acr
                    01 Dai Co Viet Road, Hanoi\acr
                    VIETNAM}
\email{tan.nguyenduy@hust.edu.vn}

%%\acr is not required (if you do not need to see a column);
%%in our style \\ makes a column automatically

\thanks{J.M. is partially supported  by the Natural Sciences and Engineering Research Council of Canada (NSERC) grant R0370A01. He gratefully acknowledges the Western University Faculty of Science Distinguished Professorship in 2020-2021 and support of Western Academy for Advanced Research as Western Fellow. N.D.T. is partially supported  by the Ministry of Education and Training of Vietnam via the project "On the arithmetic of algebraic groups and homogeneous spaces over local and global fields and their extensions"}

\subjclass[2010]{Primary 11M35, 11M06, 11B68} %Secondary is optional
\keywords{Hurwitz zeta functions, Bernoulli polynomials, Generalized Bernoulli numbers, $L$-functions}

\begin{abstract}
We present a remarkably simple and surprisingly natural interpretation of the values of zeta functions at negative integers and zero. Namely we are able to relate these values to areas related to partial sums of powers. We apply these results to further interpretations of values of $L$-functions at negative integers. We hint in a very brief way at some expected connections of this work with other current efforts to understand the mysteries of the values of zeta functions at integers. 
\end{abstract}

\maketitle
\begin{center} {\small \it To the memory of Goro Shimura with gratitude and admiration}
\end{center}

 \section{Introduction}
Ever since Euler made stunning discoveries connecting some values of zeta functions with powers of $\pi$, there has been a tremendous effort of mathematicians to comprehend well the ``true reasons behind these connections" and to further extend these results to other values and other zeta functions. This work belongs to one of the most dramatic and exciting parts of mathematics and some parts of theoretical physics. Some extraordinary progress, conjectures, and various deep concepts have been obtained. Yet, the air of tremendous mystery, excitement, and strong desire to make further progress has never been altered nor slowed down.

Let us begin with two very naive questions which a child might ask: 
\begin{enumerate}
\item How many natural numbers are there?

\item What is the sum of all natural numbers?

\end{enumerate}

The obvious answers to both questions namely: ``infinity" are correct but they lack a certain depth, subtlety, and connections with combinatorial properties of natural numbers. L. Euler was able to offer some more subtle and more interesting answers: 

\begin{enumerate}
\item $1+1+1+ \cdots =-\dfrac{1}{2}.$
\item $1+2+3+ \cdots= -\dfrac{1}{12}.$
\end{enumerate}

One has to be cautious when handling these beautiful Euler's equalities in a formal way. As we shall see below that equalities can be explained in a subtle way as values of Rieman's zeta function. However, in order to do it properly, the Riemann zeta function defined below as a convergent series for complex values $s = \sigma + \ii t$ for $\sigma > 1$ has to be analytically extended to the entire complex plane except when $s = 1$. 
Therefore one can not formally handle the equalities $(1)$ and $(2)$ in the usual way as when we consider just finite sums. For example if we just add formally both equalities we shall obtain equality: 
\[
2+ 3 + 4 +\cdots = -\dfrac{7}{12}. 
\]
However if we just subtract from the equation $(2)$ the number one we obtain 
\[
2+3 + 4+\cdots = -\dfrac{13}{12}\]  which would contradict the first equality.

Therefore below we consider rigorous proper modern interpretation of Euler's equalities using values of Riemann zeta function.

We now write the above-mentioned equalities of Euler using the values of the Riemann zeta function as: 
\begin{enumerate}
\item $\zeta(0)=-\dfrac{1}{2}.$
\item $\zeta(-1)=-\dfrac{1}{12}.$
\end{enumerate}
In fact, S. Ramanujan rediscovered these statements and for example  formula $(2)$ appeared in exactly the same form in Ramanujan's letter to G. Hardy addressed in 16 January, 1913 (see \cite[page 351]{[R]}).

Recall that the Riemann zeta function $\zeta(s)$ with $s=\sigma+ i t$ is the analytical continuation of the function 
\[ f(s)= \sum_{n=1}^{\infty} \frac{1}{n^s}, \]
with $\sigma>1$ to the whole complex plane except when $s=1$ where $\zeta(s)$ has a simple pole (see \cite{[Titchman]}).  Euler's answers are deep and the assigned numbers as  a summation of these divergent series have various interesting interpretations. For interested readers, we just refer to some articles and surveys related to these topics (see for examples  \cite{[DM]}, \cite{[Hardy]}, \cite{[Kato2]}, \cite{[zeta]}, \cite{[Kirsten]}, \cite{[LMW]}, \cite{[Rao]}).

In paper \cite{[Minac]}, Min\'a\v{c} raised and answered the following naive question:

\begin{question}
Suppose $a \in \mathbb{N}$ or $a=0$. Are there any connections between the values $\zeta(-a)$ and the partial sums 
\[ S_a(M)=\sum_{n=1}^{M-1} n^a ?\] 
\end{question} 
Here $S_a(M)$ is considered as a polynomial in $M$. 

The answer is the fact on page 1 in \cite{[Minac]} 
\[ \zeta(-a)=\int_{0}^1 S_a(x)dx .\] 

Hence it is observed that 
\[ \zeta(0)=\int_{0}^{1} (x-1)dx=-\frac{1}{2} .\] 
Thus $\zeta(0)$ is equal to $(-1)$ times the area of the right triangle with  acute angles $45^{o}$ and legs of length $1$. This is because the graph of $x-1$ over the interval $[0,1]$ has the shape of a right triangle when taken together with the x-axis and y-axis  (see the displayed Figure 1 below).
%This has a delightful little interpretation as minus are of the right triangle with sides $1 \times 1$ as the graph of $x-1$ in the interval $[0,1]$ is 
\begin{figure}[!htb]
  %\centering
\includegraphics[width=0.3 \textwidth]{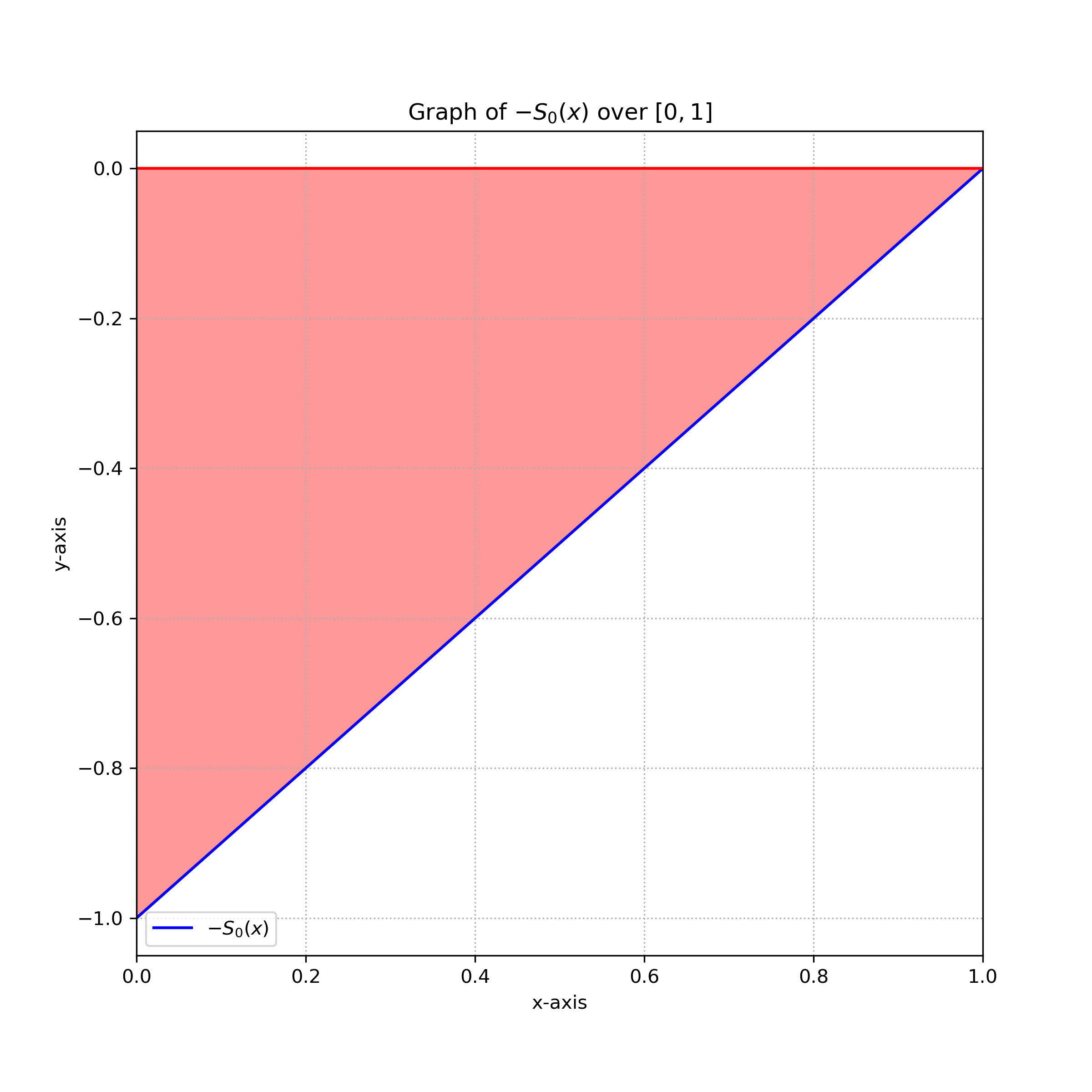} 
\caption{Graph of $x-1$ with $x \in [0,1]$}
\end{figure}

Similarly $\zeta(-1)$ is $(-1)$ time  area of the red region in figure 2 below. 
\[ \zeta(-1)=\int_{0}^1 \frac{x(x-1)}{2}dx=-\frac{1}{12} .\]

\begin{figure}[h!]
  \centering
\includegraphics[width=0.6\textwidth]{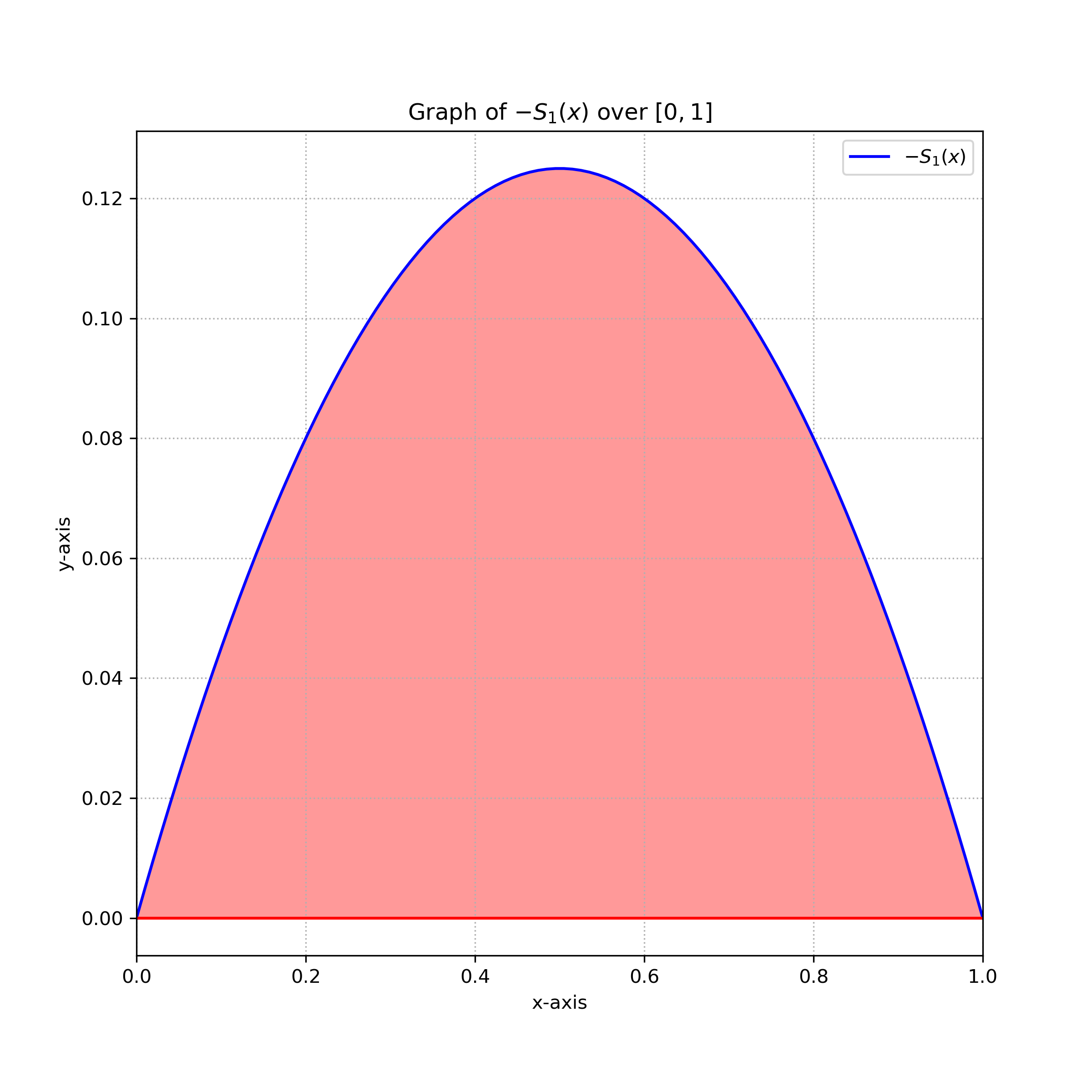} 
\caption{Graph of $\frac{x-x^2}{2}$ with $x \in [0,1]$}
\end{figure}

Observe also that 
\[ \zeta(-2)=\int_{0}^{1} \frac{x(x-1)(2x-1)}{6} dx = 0 .\] 

This fact also has a geometrical meaning as if 
\[ f(x) :=\frac{x(x-1)(2x-1)}{6}, \]
then  
\[ f(1-x)=-\frac{x(x-1)(2x-1)}{6}=-f(x) .\] 

In fact, the graph displayed function of $f(x)$ in the interval $[0,1]$ in  figure $3$, the red region bounded by the horizontal axis $x \in [0, \frac{1}{2}]$ and $f(x)$, the blue region bounded by the horizontal axis $x \in [\frac{1}{2}, 1]$ and $f(x)$ show that in  the evaluation of the integral $\int_{0}^{1} f(x)dx$, the two areas considered with proper signs will cancel each other. 

\begin{figure}[h!]
  \centering
\includegraphics[width=0.8\textwidth]{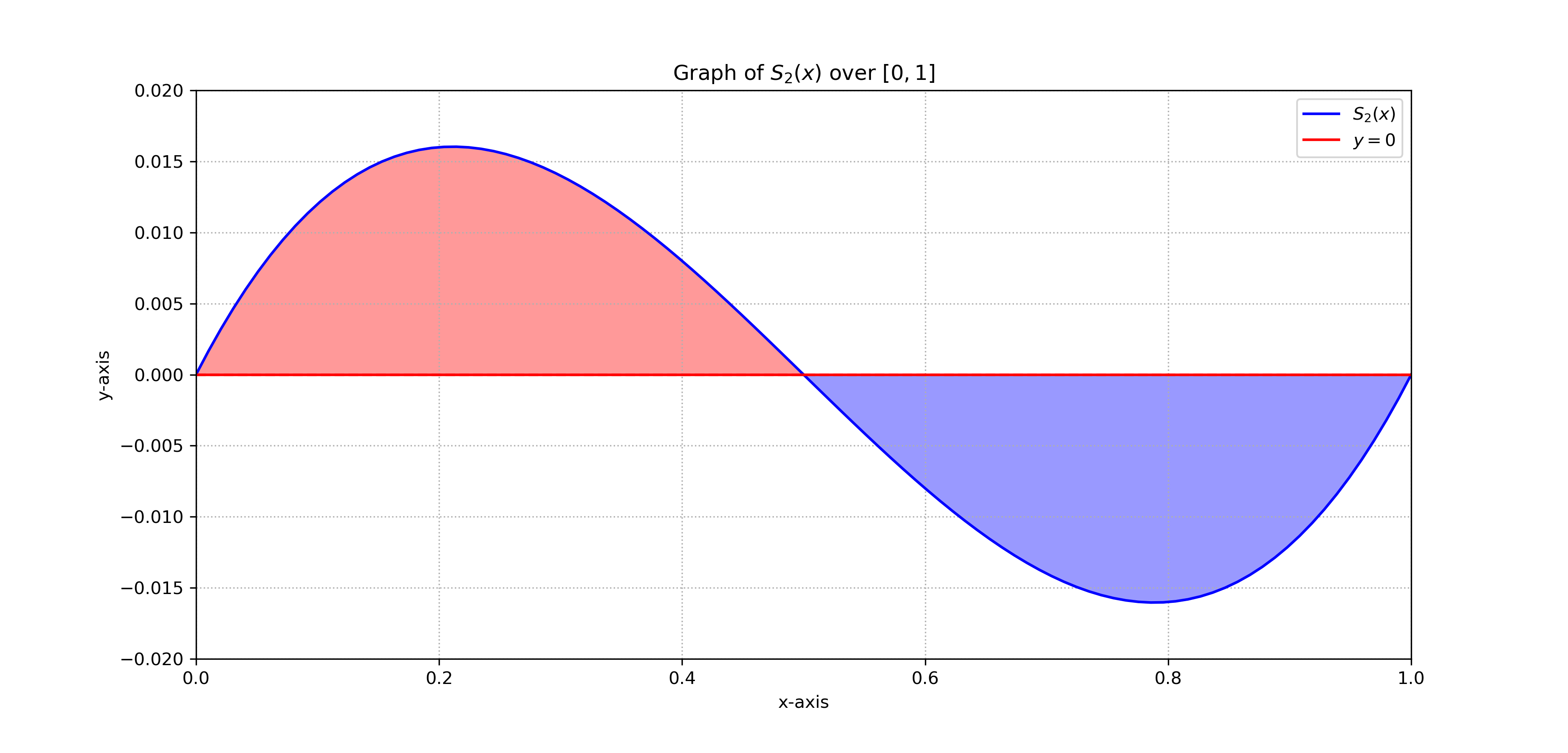} 
\caption{Graph of $\frac{x(x-1)(2x-1)}{6}$ with $x \in [0,1]$}
\end{figure}

In fact, in general for all $a \in \mathbb{N}$ 
\[ \int_{0}^{1} S_a(x)dx=\int_{0}^{1} \frac{B_{a+1}(x)-B_a(x)}{a+1}dx ,\] 

where $B_{a+1}(x)$ is the $(a+1)$-th Bernoulli polynomial (see \cite[Section 2, Chapter 1]{[Rademacher]}). Hence if $a=2n$ and $n \in \mathbb{N}$, we see that 
\[ B_{2n+1}=B_{2n+1}(1)=B_{2n+1}(0)=B_{2n+1} \left(\frac{1}{2} \right)=0 .\]
From the general property 
\[ B_{m}(1-x)=(-1)^m B_m(x) ,\] 
we see immediately  in general, in the evaluation 
\[ \zeta(-2n)=\int_{0}^{1} \frac{B_{2n+1}(x)-B_{2n+1}}{2n+1} dx ,\] 
we obtain the cancellation of two signed areas as in the special case with $\zeta(-2)$. Hence we see a nice geometric explanation $\zeta(-2n)=0$ for all $n \in \mathbb{N}$.

Because $B_{2n}(x)-B_{2n}$ does not change sign between $0$ and $1$ (see \cite[Section 3, Chapter 1]{[Rademacher]}), we also see that all values $\zeta(-1), \zeta(-3), \zeta(-5), \ldots$ are in fact signed areas of regions by the horizontal axis $x \in [0,1]$ and functions $\dfrac{B_2(x)-B_2}{2}, \dfrac{B_4(x)-B_4}{4}, \dfrac{B_{6}(x)-B_6}{6}, \cdots$ These values are also equal to $-\dfrac{B_2}{2}, -\dfrac{B_{4}}{4}, -\dfrac{B_6}{6}, \ldots$ where $B_2, B_4, B_6, \ldots$ are just Bernoulli numbers whose signs alternate.

Interested readers are encouraged to also see articles \cite{[Sury]}, \cite{[ERW]}. This paper is meant as the first paper of a series of papers related to special values of $L$-functions. It is beyond the scope of this paper to provide a comprehensive historical overview of the study of the special values of $L$-functions and their arithmetic significance. Therefore, we merely point out several highlights of these investigations which also motivate our investigations. Some of these investigations are rather sophisticated and they involve the most advanced topics of current number theory. We hope that our elementary point of view and possibly unexpected direct connections with partial sums and signed areas will bring additional insight and inspiration to understand special values of $L$-functions. We also hope to find a bridge between these elementary and more sophisticated studies of $L$-functions in the subsequent papers.

One of the strongest motivations for studying special values of $L$-functions is the class number formula for residue of the Dedekind zeta function $\zeta_K(s)$ of an algebraic number field $K$ expressed via basic invariants of $K$ like the regulator $R_K$, the class number $h_K$, the number of roots of unity $\omega_K$ contained in $K$, discriminant $D_K$,  the number of real 
places $r_1$, and the number of complex places $r_2$. This remarkable formula has the following form 

\[{\displaystyle \lim _{s\to 1}(s-1)\zeta _{K}(s)={\frac {2^{r_{1}}\cdot (2\pi )^{r_{2}}\cdot R_{K}\cdot h_{K}}{w_{K}\cdot {\sqrt {|D_{K}|}}}}} .\] 
See for example \cite[Chapter 5, Section 2.1]{[Borevic]}. This formula provides a bridge between the algebraic and the analytic world.

Observe in particular that $\zeta(s)=\zeta_{\QQ}(s)$. Because $\lim_{s \to 1}(s-1) \zeta(s)=1$ we see that the class number formula in view of $r_1=1, r_2=0, \omega_{\QQ}=2, |D_{\QQ}|=1$ amounts to $h_{\QQ}=1$ and we recover the unique factorization of integers into products of primes. In fact, these strong immediate connections between arithmetic and analysis were already used and exploited by Euler.

In the subsequent work of H. Poincar\'e, L. Mordell, E. Artin, H. Hasse, A. Weil and others, these ideas have been generalized to higher dimensional objects beginning with elliptic curves and studies of their related $L$-functions. These generalizations not only enriched and empowered number theory and its reach but they also provided a new broad perspective on the original investigations of L. Euler on special values of various $L$-functions. Using early computer experimentation, and beautiful elementary heuristic arguments, B. Birch and P. Swinnerton-Dyer proposed their remarkable conjecture relating the rank of a given elliptic curve $E$ over $\QQ$ with the Taylor expansion of $L(E,s)$ at $s=1$. We refer readers to more details of the Birch and Swinnerton-Dyer conjecture and its various refinements and developments to a beautiful short exposition of A. Wiles in \cite{[Wiles]}. There readers will also find important further references which include work of A. Beilinson, S. Bloch, G. Faltings, J. Tate, and others. Here we only want to point out that S. Bloch, A. Beilinson, and P. Deligne were able to relate $L$-functions attached to various algebraic varieties with motivic cohomology. However, their work predicts $L$-values only up to an undetermined rational factor. On the other hand, L. Euler found values of $\zeta(2n)$ for each $n \in \mathbb{N}$

\[ \zeta(2n)=(-1)^{n+1} \frac{(2 \pi)^{2n}}{2 (2n)!} B_{2n} .\] 
This formula has a considerable significance because it not only highlights the important transcendental factor $\pi^{2n}$ but also the very important rational factor of great arithmetic and combinatorial significance. In fact, Euler was able to essentially determine his ``real version of functional equation" of later Riemann's complex valued zeta functions (see \cite{[Landau2]} and \cite{[V]}).

Applying his consideration, Euler was thus able to relate values of $\zeta(2n)$ with values $\zeta(1-2n)$ we considered above. In agreement with Euler considerations, S. Bloch and K. Kato were able to propose their well-known conjecture on special values of $L$-functions (see \cite{[BK]}, \cite{[DM]},  \cite{[Geer]}, \cite{[Kato1]}, \cite{[Kato2]}, \cite{[Kato3]}, \cite{[Kurokawa]}, \cite{[LMW]}, \cite{[TungNguyen]}, \cite{[Rao]}, \cite{[Sujatha]}, \cite{[Zagier1]}, \cite{[Zagier2]}, \cite{[Zagier3]}). Namely S. Bloch and K. Kato proposed a beautiful insightful formula for these corresponding $L$-functions using suitably Tamagawa measure. The details are beyond the scope of this paper and we just refer readers to \cite{[BK]} and subsequent papers.

For us an important motivation for this project where this paper is our first contribution is that these special values of $L$-functions are related to certain volumes and that they have a deep arithmetic and combinatorial meaning.

We have been further influenced by the remarkable Elkies's paper \cite{[Elkies]} and related papers \cite{[Andre]}, \cite{[Andre2]}, \cite{[BKC]}, \cite{[Kalman]}, \cite{[Stanley1]}, and \cite{[Stanley2]} as well as other papers where the elegant and clever calculations of related polytopes is used together with very interesting combinatorial interpretation of some related numbers is used to provide quite stunning insight into some special values of $L$-functions and zeta values in general.

We now discuss very shortly the outline of our paper. In section $2$ we provide they key generalizations of some results mentioned in the Introduction from Riemann zeta values to Hurwitz zeta values. In section 3, we apply these results to get a similar insight into some values of $L$-functions at negative integers. In section $4$, we further generalize our considerations to study values of twisted $L$-functions. Finally, in section 5 we consider a special case of zeta functions namely $L(s, \chi_4)$. This is somewhat complementary insight to inspiring results in \cite{[Elkies]}. We conclude our paper with Appendix A containing some specific examples, formulas, graphs, tables as well as some further considerations and comments on a specific literature.

\section{Hurwitz zeta function}
The Hurwitz zeta function is defined for complex arguments $s$ with ${\rm Re}(s)>1$ and real $a$ with $0<a\leq 1$ by  
\[
\zeta(s,a)=\sum_{n=0}^\infty \dfrac{1}{(n+a)^s}.
\]
This series is absolutely convergent for the given values of $s$ and $a$ and can be extended to a meromorphic function defined for all $s\not=1$. The Riemann zeta function  $\zeta(s)=\zeta(s,1)$.

The Bernoulli polynomials $B_n(x)$ and Bernoulli numbers $B_n$ are defined as
\[
\dfrac{te^{xt}}{e^t-1}=\sum_{n=0}^\infty B_n(x) \dfrac{t^n}{n!}
\]
and 
\[
B_n(x) =\sum_{k=0}^n \binom{n}{k} B_k x^{n-k}.
\]

For every $n\geq 1$,  the Bernoulli polynomial $B_n(x)$ satisfies the difference equation
\[
B_n(x+1)-B_n(x)=nx^{n-1}.
\]
In particular $B_n(1)=B_n(0)$. We also have
\[
B_n^\prime(x)=nB_{n-1}(x). 
\]

Bernoulli numbers can be defined recursively by
\[
B_0=1, B_1=-\dfrac{1}{2}, B_n=\sum_{k=0}^n {n \choose k} B_k, n \geq 2.
\] 

\begin{thm}(See \cite[Theorem 12.13]{[Apostol]}) \label{Hurwitz1} 
For every integer $n\geq 0$, we have
\[
\zeta(-n,a)=-\dfrac{B_{n+1}(a)}{n+1}.
\]
\end{thm}
For any $a\in (0,1]$, $n\in \NN$ we define $S_{n,a}(x)$ to be the unique polynomial such that for every  integer $M\geq 2$ one has 
\[
S_{n,a}(M)= \sum_{k=0}^{M-2} (k+a)^n.
\]
We also set $S_n(x)= S_{n,1}(x)$, so that $S_n(x)$ is the unique polynomial sastifying that 
\[
S_n(M) =\sum_{k=1}^{M-1} k^n, \text{ for every positive integer $M$}.
\]
Existence of these polynomials follows from Proposition \ref{prop:existence} below. The uniqueness is clear because no two distinct polynomials with complex coefficients can have the same values at infinitely many arguments. Observe that our polynomials $S_{n,a}(x)$ has rational coefficients. 

\begin{prop} \label{prop:existence}
\[
S_{n,a}(x)=\dfrac{B_{n+1}(x+a-1)-B_{n+1}(a)}{n+1}.
\]
\end{prop}
\begin{proof} For every  integer $M\geq 2$, one has
\[
\begin{aligned}
S_{n,a}(M)&= \sum_{k=0}^{M-2} (k+a)^n\\
 &= \sum_{k=0}^{M-2} \dfrac{B_{n+1}(k+a+1)- B_{n+1}(k+a)}{n+1}\\ &=\dfrac{B_{n+1}(M+a-1)-B_{n+1}(a)}{n+1}.
\end{aligned}
\]
\end{proof}
We have the following theorem which generalizes the fact on \cite[page 459]{[Minac]}.
\begin{thm}
\label{thm:special value Hurwitz}
\[
\int_{1-a}^{2-a} S_{n,a}(x)dx = \zeta(-n,a).
\]
\end{thm}
We will give three different proofs for this theorem. First, let us explain our first proof using the relation between $S_{n,a}(x)$ and Bernoulli polynomials. 

\begin{proof}[First proof of Theorem~\ref{thm:special value Hurwitz}]
One has

\begin{align*}
\int_{1-a}^{2-a} S_{n,a}(x)dx &= \int_{1-a}^{2-a} \dfrac{B_{n+1}(x+a-1)-B_{n+1}(a)}{n+1} dx\\
&= \dfrac{1}{n+1}\dfrac{B_{n+2}(x+a-1)}{n+2}\bigg\rvert_{1-a}^{2-a} - \dfrac{B_{n+1}(a)}{n+1} x\bigg\rvert_{1-a}^{2-a}\\
&= \dfrac{B_{n+2}(1)-B_{n+2}(0)}{(n+1)(n+2)} - \dfrac{B_{n+1}(a)}{n+1}\\
&= \zeta(-n,a).
\qedhere
\end{align*}

\end{proof}
We provide another proof using the relation between values of the Hurwitz zeta functions and values of the classical zeta function.

By definition when $n=0$, $S_{0,a}(x)=x-1$.  In this case, the above proposition is true because 
\[ \int_{1-a}^{2-a} (x-1)dx=\frac{1}{2}-a =\zeta(0, a) .\] 
From now on, we assume $n>0$. We recall the following proposition which expresses the Hurwitz zeta functions in term of the Riemann zeta function. 
\begin{prop} [See \cite{[Kirsten]}, Theorem 2.2] 
For $0 < a \leq 1$, we have 
\[ \zeta(s,a)=\frac{1}{a^s}+\sum_{k=0}^{\infty} (-1)^k \frac{\Gamma(s+k)}{\Gamma(s) k!} a^{k} \zeta(s+k) .\] 
\end{prop} 

\begin{proof}[Second proof of Theorem~\ref{thm:special value Hurwitz}]
We recall that that $\zeta(s)$ has a simple pole at $s=1$ with residue $1$ and $\Gamma(s)$ has a simple pole at $-m$ for $m \in \NN$ with residue $\dfrac{(-1)^m}{m!}$. By plugging in $s=-n$ and using the above formula and these two facts,  we have the following relations 
\begin{equation} 
\label{first_relation} 
\zeta(-n,a)=a^{n}-\frac{a^{n+1}}{n+1}+\sum_{k=0}^{n} {n \choose k} a^k \zeta(k-n). 
\end{equation} 
Let us consider the sum 
\begin{align*} 
S_{n,a}(M) &=\sum_{m=0}^{M-2}(a+m)^n \\ 
                   &=\sum_{m=0}^{M-2} \left( \sum_{k=0}^{n} {n \choose k} a^{k} m^{n-k} \right) \\ 
                   &=\sum_{k=0}^{n} {n \choose k} a^{k} \left(\sum_{m=0}^{M-2} m^{n-k} \right) \\
                   &=a^{m} S_{0}(x)+ \sum_{k=0}^{n-1} {n \choose k} a^{k} S_{n-k}(M-1).
\end{align*} 
Since this is true for all $M \in \NN$, we conclude that 
\begin{equation} \label{eq:recursive_s_a} 
	S_{n,a}(x)=a^{n} S_{0}(x)+\sum_{k=0}^{n-1} {n \choose k} a^{k} S_{n-k}(x-1). 
\end{equation} 	
	
By integrating both sides from $1-a$ to $2-a$ we get 
\begin{equation} \label{second_relation}
\int_{1-a}^{2-a} S_{n,a}(x)dx=a^{m} \int_{1-a}^{2-a} S_0(x)+ \sum_{k=0}^{n-1} {n \choose k} a^k \int_{1-a}^{2-a} S_{n-k}(x-1) dx. 
\end{equation} 
To compute the right hand side, we need the following lemma. 
\begin{lem} 
\label{integration}
For $k \geq 1$, we have 
\[ \int_{1-a}^{2-a} S_{k}(x-1)dx=\frac{(-a)^{k+1}}{k+1}+\zeta(-k) .\] 
\end{lem} 
\begin{proof}
First, we recall that 
\[ S_{k}(x)=\frac{B_{k+1}(x)-B_{k+1}(1)}{k+1} .\] 
Therefore, we have
\begin{equation*}
\begin{split}
\int_{1-a}^{2-a} S_{n}(x-1)dx 
  &= \int_{-a}^{1-a} S_{n}(x)dx \\
  &= \int_{-a}^{1-a} \frac{B_{k+1}(x)-B_{k+1}(1)}{k+1}  \\
  &= \frac{(-a)^{k+1}}{k+1}-\frac{B_{k+1}(1)}{k+1} \\
  &=\frac{(-a)^{k+1}}{k+1}+\zeta(-k).  
 \end{split}
 \end{equation*} 
Here we use the fact that $\zeta(-k)=-\dfrac{B_{k+1}(1)}{k+1}$. 
\end{proof} 

Note that when $k=0$, the above statement is not true. Instead, we have 
\[ \int_{1-a}^{2-a} S_{0}(x)dx= \zeta(0,a)=\frac{1}{2}-a =\zeta(0)+(1-a) .\] 
By Lemma $\ref{integration}$ and the above remark, equation $\ref{second_relation}$ becomes 
\begin{align*}
 \int_{1-a}^{2-a} S_{n,a}(x)dx &= a^n(\zeta(0)+(1-a))+  \sum_{k=0}^{n-1} {n \choose k} a^{k} \left(\frac{(-a)^{n-k+1}}{n-k+1}+\zeta(n-k) \right)\\ 
                                               &=a^n(1-a)+ \sum_{k=0}^{n} {n \choose k}  a^k \zeta(k-n)+\frac{a^{n+1}}{n+1}  \sum_{k=0}^{n-1} (-1)^{n-k+1} {{n+1} \choose k}.    
\end{align*} 
Note that 
\[ \sum_{k=0}^{n+1} (-1)^{n-k+1} {{n+1} \choose k} =(1-1)^{n+1}=0.\] 
Therefore
\[ \sum_{k=0}^{n-1} (-1)^{n-k+1} {{n+1} \choose k} = -(-(n+1)+1)=n .\]  

 Consequently, we have 

\begin{align*}
\int_{1-a}^{2-a} S_{n,a}(x)dx &=a^n(1-a)+\frac{na^{n+1}}{n+1} +\sum_{k=0}^{n} {n \choose k}  a^k \zeta(k-n)\\ 
                                              &=a^n -\frac{a^{n+1}}{n+1} +\sum_{k=0}^{n} {n \choose k}  a^k \zeta(k-n). 
\end{align*} 
Compare this equation and equation $\ref{first_relation}$, we conclude that 
\[
 \zeta(-n,a)=\int_{1-a}^{2-a} S_{n,a}(x)dx .
\qedhere
\] 
\end{proof}

Finally, we provide another proof without using the Bernoulli polynomials. Instead, we use the relation between values of the Hurwitz zeta functions to prove this theorem.  Note that our approach is very close to the second approach discussed in \cite{[Minac]}. 

\begin{proof}[Third proof of Theorem~\ref{thm:special value Hurwitz}]
For each $n \geq 1$ we have 
\begin{align*} 
\frac{1}{(n+a-1)^{s-1}}-\frac{1}{(n+a)^{s-1}} &=\frac{1}{(n+a)^s} \left( (1-\frac{1}{n+a})^{-(s-1)}-1 \right) \\
                                                                    &=\frac{1}{(n+a)^{s-1}} \left(\frac{s-1}{(n+a)}+ \frac{(s-1)s}{2!} \frac{1}{(n+a)^{2}}+\cdots \right)\\
                                                                    &=\sum_{k=0}^{\infty} \frac{\Gamma(s+k)}{\Gamma(s-1)(k+1)!} \frac{1}{(a+n)^{s+k}}.
\end{align*} 
Summing over all $n$ we have 
\begin{align*}
 \frac{1}{a^{s-1}} &=\sum_{n=1}^{\infty} \left[ \sum_{k=0}^{\infty} \frac{\Gamma(s+k)}{\Gamma(s-1)(k+1)!} \frac{1}{(a+n)^{s+k}} \right]\\ 
                           &=\sum_{k=0}^{\infty}  \frac{\Gamma(s+k)}{\Gamma(s-1)(k+1)!}  \left[ \sum_{n=1}^{\infty} \frac{1}{(a+n)^{s+k}} \right]\\
                           &=\sum_{k=0}^{\infty}  \frac{\Gamma(s+k)}{\Gamma(s-1)(k+1)!} (\zeta(s+k,a)-\frac{1}{a^{s+k}}).
\end{align*} 
We remark that $\zeta(s,a)$ is a meromorphic function with the only simple pole at $s=1$. Furthermore, the residue at $s=1$ is $1$. Furthermore, for each $m \in \NN$, $\Gamma(s)$ has simple pole at $-m$ with residue $\dfrac{(-1)^m}{m!}$. By plugging in $s=-n$ and using these two facts, we have the following 
\begin{equation} \label{relation_1}
a^{n+1}=\sum_{k=0}^n (-1)^{k+1} {{n+1} \choose {k+1}} (\zeta(k-n,a)-a^{n-k})+\frac{(-1)^{n+1}}{n+2} . 
\end{equation} 
Next, we show that the integrals $\int_{1-a}^{2-a} S_{n,a}(x)dx$ satisfy a similar relation. To do so, we use the following collapsing sum
\begin{align*}
(M+a-2)^{n+1} &=a^{n+1}+\sum_{k=1}^{M-2}[(k+a)^{n+1}-(k+a-1)^{n+1}]\\
                        &=a^{n+1}+\sum_{k=1}^{M-2} \left[ \sum_{m=1}^{n+1} {{n+1} \choose m} (-1)^{m+1} (k+a)^{n+1-m} \right]\\
                        &=a^{n+1}+ \sum_{m=1}^{n+1} \left[ (-1)^{m+1} {{n+1} \choose m} \sum_{k=1}^{M-2} (k+a)^{n-m+1} \right]\\
                        &=a^{n+1}+\sum_{m=1}^{n+1} (-1)^{m+1} {{n+1} \choose m} (S_{n-m+1}(M)-a^{n-m+1})\\
                        &=a^{n+1}+\sum_{k=0}^{n} (-1)^{k} {{n+1} \choose k+1} (S_{n-k}(M)-a^{n-k}).                         
\end{align*} 
The last equality is obtained by substituting $m=k+1$.

By integrating both sides of this equation, we have
\begin{equation}  \label{relation_2}
\int_{1-a}^{2-a} (x+a-2)^{n+1}dx=a^{n+1}+\sum_{k=0}^{n} (-1)^{k} {{n+1} \choose k+1} \left(\int_{1-a}^{2-a} S_{n-k}(x)dx -a^{n-k} \right).
\end{equation} 
We have
\[ \int_{1-a}^{2-a} (x+a-2)^{n+1}dx=\int_{-1}^{0} x^{n+1}dx=\frac{(-1)^{n+1}}{n+2} .\] 
Therefore, equation $\ref{relation_2}$ can be written as 
\begin{equation} \label{relation_3}
a^{n+1}=\sum_{k=0}^{n} (-1)^{k+1} {{n+1} \choose k+1} \left(\int_{1-a}^{2-a} S_{n-k}(x)dx -a^{n-k} \right)+\frac{(-1)^{n+1}}{n+2}.
\end{equation} 
We know that $\int_{1-a}^{2-a}S_{0,a}(x)dx=\dfrac{1}{2}-a=\zeta(0,a)$. By equations $\ref{relation_1}$ and $\ref{relation_3}$, we can conclude that 
\[ \int_{1-a}^{2-a} S_{a,n}(x)dx=\zeta(-n,a) .
\qedhere
\] 
\end{proof}

As a corollary of Theorem~\ref{thm:special value Hurwitz}, one obtains
\[
\zeta(-n)=\zeta(-n,1) =\int_{0}^1 S_{n,1}(x)dx= \int_{0}^1 S_n(x) dx. 
\]

Let us provide some concrete examples for Theorem \ref{thm:special value Hurwitz}. 
\begin{examples} \label{S_a} Let $0<a \leq 1$. Then we have 
\[ S_{0,a}(x)=x-1 .\] 
\[ S_{1,a}(x)=\sum_{k=0}^{x-2} (a+k)=\frac{(2a+x-2)(x-1)}{2}.\]
Inductively, we can compute $S_{n,a}(M)$ using the  recursive relation \ref{eq:recursive_s_a} 
\[ S_{n,a}(x)=a^{n} S_{0}(x)+ \sum_{k=0}^{n-1} {n \choose k} a^{k} S_{n-k}(x-1).\] 
For example, for $n=2$ we have 
\[ S_{2,a}(x)=\sum_{k=0}^{x-2} (a+k)^2=(x-1)a^2+a(x-2)(x-1)+\frac{(x-2)(x-1)(2x-3)}{6} .\] 
For $n=3$ we have 
\[ S_{3,a}(x)=a^3(x-1)+3a^2 \frac{(x-1)(x-2)}{2}+3a \frac{(x-2)(x-1)(2x-3)}{6}+\frac{((x-1)(x-2))^2}{4} .\] 
\end{examples}

\begin{examples}
Let us consider the special case when $a=\frac{1}{2}$. In this case, the Hurwitz zeta function is given by 
\[ \zeta(s,\frac{1}{2})=\sum_{n=0}^{\infty} \frac{1}{(n+\frac{1}{2})^s}=2^{s} \sum_{n=0}^{\infty} \frac{1}{(2n+1)^s} \]
By the above calculations, we have 
\[ S_{0, \frac{1}{2}}(x)=x-1, \]
and 
\[ S_{1, \frac{1}{2}}(x)=\frac{(x-1)^2}{2}. \]
By Theorem \ref{thm:special value Hurwitz}, we have 
\[ \zeta(0, \frac{1}{2})=\int_{\frac{1}{2}}^{\frac{3}{2}} (1-x)dx=0.\]
We also have 
\[ \zeta(1, \frac{1}{2})=\int_{\frac{1}{2}}^{\frac{3}{2}} \frac{(x-1)^2}{2}dx=\frac{1}{24}.\]

\end{examples}

\begin{cor} If $n=2k+1$ is an odd positive integer, then
\[
\zeta(-2k-1)=(-1)^{k+1} (area),
\]
here $(area)$ is the area of a region bounded by the $x$-axis and the graph of the function $\dfrac{B_{2k+2}(x)-B_{2k+2}}{2k+2}$ on the interal $[0,1]$. 
\end{cor}
\begin{proof}
One has
\[
\zeta(-2k-1)= \int_{0}^1 S_{2k+1}(x) dx= \int_{0}^1 \dfrac{B_{2k+2}(x)-B_{2k+2}}{2k+2} dx.
\]
On the other hand, 
\[\zeta(-2k-1)=-\dfrac{B_{2k+2}}{2k+2}=(-1)^{k+1} \dfrac{(-1)^kB_{2k+2}}{2k+2}.\] 
Hence \[\dfrac{(-1)^kB_{2k+2}}{2k+2}=(-1)^{k+1}\zeta(-2k-1))= \int_{0}^1 (-1)^{k+1}\dfrac{B_{2k+2}(x)-B_{2k+2}}{2k+2} dx.\]
By \cite[Page 8]{[Rademacher]}, $B_{2k+2}(x)-B_{2k+2}$ does not change sign between $0$ and $1$
 and $(-1)^kB_{2k+2}>0$.
Therefore 
 $\phi_{k}(x):=(-1)^{k+1}(B_{2k+2}(x)-B_{2k+2})/(2k+2)>0$ on $(0,1)$.
Hence $\int_0^1 \phi_k(x)dx$ is the area of the region bounded by the $x$-axix and the graph of $\phi_k(x)$ on the interval $[0,1]$,  and 
\[\zeta(-2k-1)=(-1)^{k+1}\int_0^1 \phi_k(x)dx=(-1)^{k+1}(area).\qedhere\]
\end{proof}

\begin{prop} 
\[
S_n(x) + (-1)^n S_n(1-x)=0.
\]
\end{prop}
\begin{proof}
We first observe that $S_n(x+1)-S_n(x)=x^n$ and $S_n(1)=S_n(0)=0$. 
 Now let us consider the polynomial
\[
p_n(x)=S_n(x) + (-1)^n S_n(1-x).
\]
We have
\[
\begin{aligned}
p_n(x+1)-p_n(x)&= [S_n(x+1)-S_n(x)]- (-1)^n [S_n(-x)-  S_n(1-x)]\\
&=x^n -(-1)^n(-x)^n\\
&=0.
\end{aligned}
\]
Hence $p_n(x)$ must be a constant. Therefore
\[
p_n(x)=S_n(0)+(-1)^nS_n(1)=0.
\qedhere
\]
\end{proof}

\begin{prop} Suppose $n$ is an even positive integer. Then
\[
\zeta(-n)=\int_{0}^1 S_n(x) dx=0\] 
and
\[
\int_{0}^{1/2} S_n(x) dx =\dfrac{(2^{-n-1}-2)}{(n+1)(n+2)}B_{n+2}.
\]
\end{prop}
\begin{proof} 
Since $n$ is even, $S_n(x)+S_n(1-x)=0$. Hence
\[
\begin{aligned}
 \zeta(-n)&=  \int_{0}^1 S_n(x) dx  = \int_{0}^{1/2} S_n(x) dx + \int_{1/2}^1 S_n(x) dx\\
&=\int_{0}^{1/2} S_n(x) dx - \int_{1/2}^{0} S_n(1-x) dx\\
&= \int_{0}^{1/2} S_n(x) dx + \int_{0}^{1/2} S_n(1-x) dx\\
&=0.
\end{aligned}
\]

For the second statement, one has
\[
\begin{aligned}
\int_{0}^{1/2} S_n(x)dx &= \int_0^{1/2} \dfrac{B_{n+1}(x)-B_{n+1}}{n+1} dx\\
&= \dfrac{1}{n+1}\dfrac{B_{n+2}(x)}{n+2}\bigg\rvert_{0}^{1/2} \\
&= \dfrac{B_{n+2}(1/2)-B_{n+2}(0)}{(n+1)(n+2)}\\
&= \dfrac{(2^{-n-1}-2)}{(n+1)(n+2)}B_{n+2},
\end{aligned}
\]
here we have used the formula that $B_{k}(1/2)=(2^{1-k}-1)B_k$ for every $k\geq 0$.
\end{proof}

\section{Special values of  $L$-functions}
Suppose $\chi$ is a Dirichlet character of conductor $k$. The $L$-fuction attached to $\chi$ is defined by
\[
L(s,\chi)=\sum_{n=1}^\infty \dfrac{\chi(n)}{n^s},\quad {\rm Re}(s)>1.
\]
The function $L(s,\chi)$ can be continued analytically to the whole complex plane, except for a simple pole at $s=1$ when $\chi=1$.
One has
\[
L(s,\chi)=\sum_{r=1}^k\chi(r) k^{-s} \zeta(s,\dfrac{r}{k}).
\]
We define
\[
S_{n,\chi}(x) =k^n\sum_{r=1}^k \chi(r) S_{n,r/k}(x+1-r/k).
\]
\begin{lem} If $\chi$ is nontrivial then 
$S_{n,\chi}(x)=-\dfrac{B_{n+1,\chi}}{n+1}$.
\end{lem}
\begin{proof}
One has
\[
\begin{aligned}
S_{n,\chi}(x) &=k^n\sum_{r=1}^k \chi(r) S_{n,r/k}(x+1-r/k)\\
&= k^n\sum_{r=1}^k \chi(r) \dfrac{B_{n+1}(x)-B_{n+1}(r/k)}{n+1}\\
&=k^n\sum_{r=1}^k \chi(r) \dfrac{B_{n+1}(x)}{n+1} -k^n\sum_{r=1}^k \chi(r) \dfrac{B_{n+1}(r/k)}{n+1}
=-\dfrac{B_{n+1,\chi}}{n+1},
\end{aligned}
\]
here we used the formula that  $\displaystyle B_{n+1,\chi}=k^n\sum_{r=1}^k \chi(r) B_{n+1}(r/k)$ and that for non-trivial $\chi$, $\sum\limits_{r=1}^{k} \chi(r)=0$. 
\end{proof} 
\begin{cor} \label{cor: twisted_chi}
\[
L(-n,\chi) = \int_{0}^1 S_{n,\chi}(x)dx=-\dfrac{B_{n+1,\chi}}{n+1}.
\]
\end{cor}
\begin{proof} We have
\begin{align*}
L(-n,\chi) & =k^n \sum_{r=1}^k \chi(r)\zeta(-n,r/k)\\
&= k^n \sum_{r=1}^k \chi(r) \int_{1-r/k}^{2-r/k} S_{n,r/k}(x)dx\\
&= k^n \sum_{r=1}^k \chi(r) \int_{0}^{1} S_{n,r/k}(x+1-r/k)dx\\
&=    \int_{0}^{1} k^n\sum_{r=1}^k \chi(r) S_{n,r/k}(x+1-r/k)dx\\
&=\int_{0}^1  S_{n,\chi}(x)dx.
\qedhere
\end{align*}
\end{proof}

Recall that $\chi$ is a Dirichlet character of conductor $k$. We define $P_{n,\chi}(x)$ to be the unique polynomial such that for every  integer $M\geq 1$ one has
\[
P_{n,\chi}(M)=\sum_{r=1}^{Mk} \chi(r) r^{n}.
\]

\begin{lem}
\[
P_{n,\chi}(x)=\dfrac{B_{n+1,\chi}(xk)-B_{n+1,\chi}}{n+1}.
\]
\end{lem}
\begin{proof} For every  integer $M\geq 2$, one has
\[
P_{n,\chi}(M)= \sum_{r=0}^{Mk} \chi(r) r^n
= \dfrac{B_{n+1,\chi}(Mk)-B_{n+1,\chi}}{n+1}.
\qedhere
\]
\end{proof}
Recall that character $\chi$ is said to be even (respectively, odd) if $\chi(-1)=\chi(-1)$ (respectively, $\chi(-1)=-\chi(1)$).

\begin{lem} Suppose either $\chi$ is even and $n$ even or $\chi$ is odd and $n$ is odd. Then $B_{n,\chi}(k/2)=B_{n,\chi}(-k/2)$.
\end{lem}
\begin{proof} For any $x$, one has
\[
B_{n,\chi}(x+k)-B_{n,\chi}(x)= n\sum_{r=1}^k\chi(r)(r+x)^{n-1}.
\]
By substituting $x=-k/2$, one obtains
\[
B_{n,\chi}(\dfrac{k}{2})-B_{n,\chi}(-\dfrac{k}{2})= n\sum_{r=1}^k\chi(r)(r-\dfrac{k}{2})^{n-1}.
\]
We treat the case that $\chi$ is even and $n$ is even. The other case can be treated similarly.
Since $\chi$ is even and $n$ is even, for each $r$ coprime to $k$, one has
\[
\begin{aligned}
\chi(r) (r-\dfrac{k}{2})^{n-1}+\chi(k-r)(\dfrac{k}{2}-r)^{n-1}&=\chi(r) (r-\dfrac{k}{2})^{n-1}+\chi(-r)(\dfrac{k}{2}-r)^{n-1}\\
&= (\chi(r)-\chi(-r)) (r-\dfrac{k}{2})^{n-1}=0.
\end{aligned}
\]
Thus 
\[
B_{n,\chi}(\dfrac{k}{2})-B_{n,\chi}(-\dfrac{k}{2})= n\sum_{r=1}^k\chi(r)(r-\dfrac{k}{2})^{n-1}=0.
\qedhere
\]
\end{proof}
\begin{prop} Suppose either $\chi$ is even and $n$ even or $\chi$ is odd and $n$ is odd. We have
\[
\int_{-1/2}^{1/2} P_{n,\chi}(x)dx = L(-n,\chi).
\]
\end{prop}
\begin{proof} One has
\begin{align*}
\int_{-1/2}^{1/2} P_{n,\chi}(x)dx&= \int_{-1/2}^{1/2} \dfrac{B_{n+1,\chi}(xk)-B_{n+1,\chi}}{n+1}dx\\
&=  \int_{-k/2}^{k/2} \dfrac{B_{n+1,\chi}(y)dy }{k(n+1)}  -\dfrac{B_{n+1,\chi}}{n+1}\\
&= \dfrac{1}{k(n+1)}  \dfrac{B_{n+2,\chi}(y)}{n+2}\bigg\rvert_{-k/2}^{k/2} -\dfrac{B_{n+1,\chi}}{n+1}\\
&= \dfrac{B_{n+2,\chi}(k/2)-B_{n+2,\chi}(-k/2) }{k(n+1)(n+2)} - \dfrac{B_{n+1,\chi}}{n+1}\\
&=L(-n,\chi).
\qedhere
\end{align*}

\end{proof}

%%%%%%%%%%%%%%%%%%%%%%%%%%%%%%%%%%%%%%%%%%%%%%%%%%%%%%%%%%

\section{Special values of twisted $L$-functions}
Suppose $\chi$ is a Dirichlet character of conductor $k$. For each $0< a \leq 1$, the $L$-fuction attached to $\chi$ is defined by
\[
L(s,a,\chi)=\sum_{n=0}^\infty \dfrac{\chi(n+1)}{(n+a)^s},\quad {\rm Re}(s)>1.
\]
Note that when $a=1$, this definition agrees with the classical $L$-function associated with $\chi$.

We have 
\begin{align*}
L(s,a, \chi) &=\displaystyle{\sum_{n=0}^\infty \dfrac{\chi(n+1)}{(n+a)^s}}\\
	          &=\sum_{r=0}^{k-1} \left(\sum_{n \equiv r (k) }^{\infty} \frac{\chi(n+1)}{(n+a)^s} \right)\\
		 &=\sum_{r=0}^{k-1} \left(\sum_{m=0}^{\infty} \frac{\chi(mk+r+1)}{(mk+r+a)^s} \right)\\
		 &=\sum_{r=0}^{k-1} \chi(r+1)k^{-s} \left(\sum_{m=0}^{\infty} \frac{1}{(m+\frac{r+a}{k})^s} \right)\\
		 &=\sum_{r=0}^{k-1} \chi(r)k^{-s} \zeta(s, \frac{r+a}{k})\\
		 &=\sum_{r=1}^{k} \chi(r)k^{-s} \zeta(s, \frac{r+a-1}{k}).
\end{align*} 
From this expression, we can see that the function $L(s, a, \chi)$ can be continued analytically to the whole complex plane, except for a simple pole at $s=1$ when $\chi=1$.
We define
\[
S_{n,a, \chi}(x) =k^n\sum_{r=1}^k \chi(r) S_{n,\frac{r+a-1}{k}}(x+1-\frac{r+a-1}{k}).
\]
We also define the following modified generalized Bernoulli polynomials $\tilde{B}_{n, \chi}(x)$
\[ \tilde{B}_{n,\chi}(x)=B_{n,\chi}(x-1), \]
where $B_{n,\chi}(x)$ is the generalized Bernoulli polynomial defined in \cite[Equation 1.2.7]{[Dilcher]}. Our justification for this modified definition is that when $\chi$ is the trivial character, we have $\tilde{B}_{n,\chi}(x)=B_{n}(x)$. We have the following lemma.

\begin{lem} If $\chi$ is nontrivial then 
$S_{n,a, \chi}(x)=-\dfrac{\tilde{B}_{n+1,\chi}(a)}{n+1}.$
\end{lem}
\begin{proof}
One has
\[
\begin{aligned}
S_{n,a,\chi}(x) &=k^n\sum_{r=1}^k \chi(r) S_{n,\frac{r+a-1}{k}}(x+1-\frac{r+a-1}{k})\\
&= k^n\sum_{r=1}^k \chi(r) \dfrac{B_{n+1}(x)-B_{n+1}(\frac{r+a-1}{k})}{n+1}\\
&=k^n\sum_{r=1}^k \chi(r) \dfrac{B_{n+1}(x)}{n+1} -k^n\sum_{r=1}^k \chi(r) \dfrac{B_{n+1}(\frac{r+a-1}{k})}{n+1}
=-\dfrac{\tilde{B}_{n+1,\chi}(a)}{n+1},
\end{aligned}
\]
here we used the formula that  $\tilde{B}_{n+1,\chi}(a)=k^n\sum\limits_{r=1}^k \chi(r) B_{n+1}(\frac{r+a-1}{k})$ (see \cite[Equation 1.2.14]{[Dilcher]}).
\end{proof} 
%begin{rmk}
%Our definition of $B_{n+1, \chi}(x)$ is a little bit different from the one given in \cite[formula 1.2.14]{[Dilcher]} by a shift $1$. We prefer our definition  because when $\chi=1$, we have $B_{n+1,\chi}(x)=B_{n+1}(x)$. While when we would use the definition in Dilcher's book, we would obtain $B_{n+1,\chi}(x)=B_{n+1}(x+1)$.
%\end{rmk} 

\begin{cor} \label{zeta_value_1} Let $\chi$ be a non-trivial primitive Dirichlet character. We have 
\[
L(-n,a, \chi) = \int_{0}^1 S_{n,\chi}(x)dx=-\dfrac{\tilde{B}_{n+1,\chi}(a)}{n+1}.
\]
\end{cor}
\begin{proof} 
We have
\[
\begin{aligned}
L(-n,a, \chi) & =k^n \sum_{r=1}^k \chi(r)\zeta(-n,\frac{r+a-1}{k})\\
&= k^n \sum_{r=1}^k \chi(r) \int_{1-\frac{r+a-1}{k}}^{2-\frac{r+a-1}{k}} S_{n,\frac{r+a-1}{k}}(x)dx\\
&= k^n \sum_{r=1}^k \chi(r) \int_{0}^{1} S_{n,\frac{r+a-1}{k}}(x+1-\frac{r+a-1}{k})dx\\
&=    \int_{0}^{1} k^n\sum_{r=1}^k \chi(r) S_{n,\frac{r+a-1}{k}}(x+1-\frac{r+a-1}{k})dx\\
&=\int_{0}^1  S_{n,\chi,a}(x)dx.
\end{aligned}
\]
\end{proof}
\section{ The zeta function $L(s, \chi_{4})$} 
In this section, we use the studies in the previous sections to  study the zeta function $L(s,\chi_{4})$ where $\chi_{4}$ is the quadratic character of conductor $4$ given by 
\[ \chi(m) = \begin{cases} 0 &\mbox{if } 2|m   \\
1 & \mbox{if } m \equiv 1 \pmod{4} \\ 
-1 &\mbox{if }  m \equiv 3 \pmod{4}. \end{cases} \] 
\subsection{The first integral representation of $L(-n,\chi_4)$} 
By Corollary \ref{cor: twisted_chi}, we have 
\[ L(-2n, \chi_4)=-\frac{B_{2n+1, \chi_4}}{2n+1}= -\frac{4^{2n}}{2n+1} \left(B_{2n+1}\left(\frac{1}{4} \right)-B_{2n+1}\left(\frac{3}{4} \right) \right).\] 
By \cite[page 7, formula (2.8)]{[Rademacher]}), we have $B_{n}(x)=(-1)^{n} B_{n}(1-x)$, we have 
\[ B_{2n+1}\left(\frac{3}{4} \right)=-B_{2n+1}\left(\frac{1}{4} \right) .\] 
Therefore, we have 
\[ L(-2n, \chi_4)=-\frac{4^{2n+1} B_{2n+1} \left(\frac{1}{4} \right)}{2(2n+1)} .\] 
By Theorem \ref{Hurwitz1}, we know that 
\[ \zeta(-2n, \frac{1}{4})=-\frac{B_{2n+1} \left(\frac{1}{4} \right)}{2n+1} .\] 
Combining these two facts we have 
\[ L(-2n, \chi_4)=\frac{4^{2n+1}}{2} \zeta(-2n, \frac{1}{4}) .\] 
By Theorem \ref{thm:special value Hurwitz}, we have the following 
\begin{prop} \label{first_integral_chi_4} 
Let $n$ be a positive integer. Then 
\begin{enumerate}
\item If $n$ is odd then $L(-n, \chi_4)=0$. 
\item If $n$ is even then 
\[ L(-n, \chi_4)=\frac{4^{n+1}}{2} \int_{\frac{3}{4}}^{\frac{7}{4}} S_{n, \frac{1}{4}}(x)dx .\] 
\end{enumerate} 
\end{prop} 
\begin{examples}
Let us give a concrete example for Proposition \ref{first_integral_chi_4}. By Example \ref{S_a}, we have
\[ S_{2,\frac{1}{4}}(x)=\frac{1}{16}(x-1)+\frac{1}{4}(x-1)(x-2)+  \frac{(x-1)(x-2)(2x-3)}{6}.  \]
Therefore, by the above proposition we have 
\[ L(-2, \chi_4)=\frac{4^3}{2} \int_{\frac{3}{4}}^{\frac{7}{4}} \left[\frac{1}{16}(x-1)+\frac{1}{4}(x-1)(x-2)+  \frac{(x-1)(x-2)(2x-3)}{6} \right]dx=-\frac{1}{2}.  \]

\end{examples} 

\subsection{The second integral representation of $L(-n, \chi_4)$}
In this subsection, we provide another integral representation for $L(-n, \chi_4)$. Our presentation here is strongly influenced by \cite{[Shimura1]}. First, let us recall the Lerch zeta function and its special values studied in \cite{[Shimura1]}. \label{thm: Lerch}
\[\zeta(s, a, \gamma)=\sum_{n=0}^{\infty} \frac{\gamma^n}{(n+a)^s}, \]
where $0<a \in \RR$ and $0<|\gamma| \leq 1$. When $\gamma=\exp(-2 \pi i \alpha)$ with $\alpha \in \RR \backslash \ZZ$,  we have the following theorem. 
\begin{thm} \cite[Theorem 0.2]{[Shimura1]} 
For $0<k \in \ZZ$, $\Re(a)>0$, the value $\zeta(1-k, a, \gamma)$ is given by 
\[ \zeta(1-k, a, \gamma)=E_{c,k-1}(a)/(1+c^{-1}), \]
where $c=\gamma^{-1}$ and $E_{c,k-1}(t)$ is the Euler polynomial defined by \cite[Equation 0.3]{[Shimura1]}. 
\end{thm} 
When $c=1$ (equivalently $\gamma=-1)$, $E_{1,n}(t)$ is the classical Euler polynomial of degree $n$. Here are some important properties of Euler polynomials that we will use later on (see \cite[Page 25]{[Shimura2]}). 
\begin{equation} \label{Euler1}
E_{c,n}=2^n E_{c,n} \left(\frac{1}{2} \right). 
\end{equation} 
\begin{equation} \label{Euler2}
(d/dt) E_{c,n}(t)=n E_{c, n-1}(t) \quad (n>0).
\end{equation} 
\begin{equation} \label{Euler3}
E_{c,n}(t+1)+cE_{c, n}(t)=(1+c)t^n.
\end{equation} 
We have the following identity  
\begin{equation} \label{Lerch-chi_4}
L(s, \chi_4)=\sum_{n=0}^{\infty} \frac{(-1)^n}{(2n+1)^s}=\frac{1}{2^s} \sum_{n=0}^{\infty} \frac{(-1)^n}{(n+\frac{1}{2})^s} =2^{-s} \zeta(s, \frac{1}{2}, -1). 
\end{equation} 
By Theorem \ref{thm: Lerch} we have 
\[ \zeta(-n, \frac{1}{2}, -1)=\frac{E_{1, n}(\frac{1}{2})}{2} .\] 
By equation \ref{Lerch-chi_4}, we deduce that 
\[ L(-n, \chi_4)=\frac{2^n E_{1,n}(\frac{1}{2})}{2}=\frac{E_{1,n}}{2} .\] 
Let us now give an integral representation of $L(-n, \chi_4)$. From  equation \ref{Euler3}, we can see that 
\[ E_{1,n}(t+2)-E_{1,n}(t)=2((t+1)^n-t^n) .\] 
In particular, $E_{1,n}(2)-E_{1,n}(0)=2$ for all $n$. By equation \ref{Euler2} we have 
\begin{equation*}
\int_{0}^{2} E_{1,n}(x)dx =\left[ \frac{E_{1,n+1}(x)}{n+1} \right]_{0}^{2}=\frac{E_{1, n+1}(2)-E_{n+1}(0)}{n+1}=\frac{2}{n+1}. 
\end{equation*} 
Consequently, we have 
\begin{align*}
-\frac{1}{4} \int_{0}^{2} (E_{1, n}(x)-E_{1,n}-\frac{1}{n+1})dx &=\frac{E_{1,n}}{2}-\frac{1}{4} \left[ \int_{0}^{2} E_{1, n}(x)dx-\frac{2}{n+1} \right] \\
                                                                                           &=\frac{E_{1,n}}{2}= L(-n,\chi_4).
\end{align*} 
In summary, we have just proved the following
\begin{prop}
For all $n \geq 0$ we have 
\[ L(-n, \chi_4)=-\frac{1}{4} \int_{0}^{2} (E_{1, n}(x)-E_{1,n}-\frac{1}{n+1})dx .\] 
\end{prop} 

\appendix
\section{Some formulas and figures}
In this appendix, we provide the formulas of $S_n(x)$ for small $n$ and their graphs on the interval $[0,1]$. First, we have the following formulas for $S_n(x)$ with $1 \leq n \leq 6$.
\[ S_1(x)=\frac{(x-1)x}{2}, S_2(x)=\frac{(x-1)x(2x-1)}{6}, \] 
\[ S_{3}(x)=\frac{(x(x-1))^2}{4}, S_{4}(x)=\frac{1}{6} (x^6-3x^5+ \frac{5}{2} x^4-\frac{1}{2}x^2),\] 
\[ S_5(x)=\frac{1}{6} (x^6-3x^5+ \frac{5}{2} x^4-\frac{1}{2} x^2,S_6=\frac{1}{7}(x^7-\frac{7}{2}x^6+\frac{7}{2} x^5-\frac{7}{6}x^3+\frac{1}{6}x). \] 

We plot the graphs of $S_{n}(x)$ for $n \in \{1,3,5 \}$. 
\begin{center}
\includegraphics[scale=0.4]{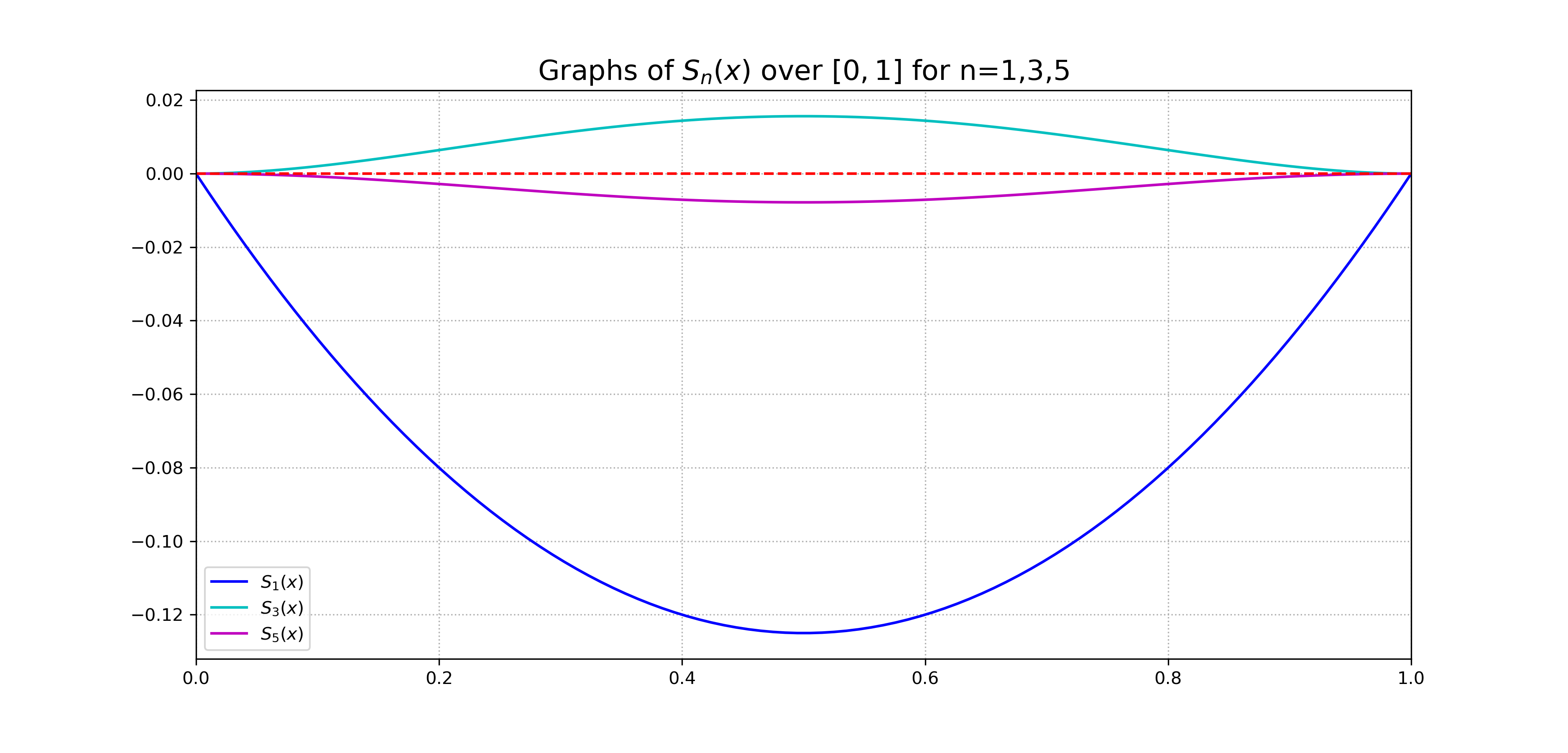} 
\end{center} 
\vspace{0.5 cm}
Here are the graphs of $S_n(x)$ for $n \in \{2,4,6 \}$. 
\begin{center}
\includegraphics[scale=0.4]{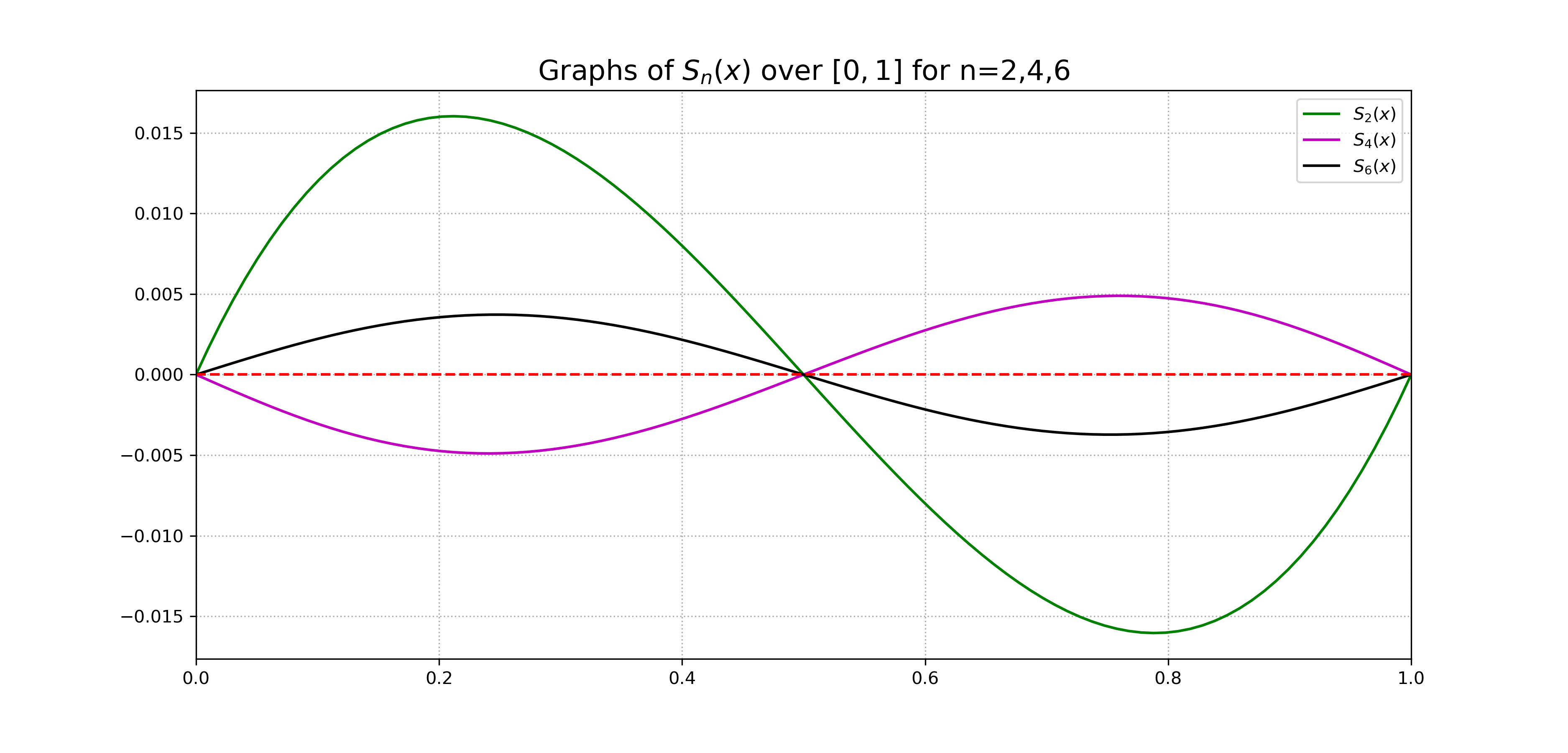} 
\end{center} 

Finally, we plot the region bounded by $S_{n}(x)$ and $y=0$ over $[0,\frac{1}{2}]$.

\hspace*{-1.0in}
\includegraphics[width=20cm]{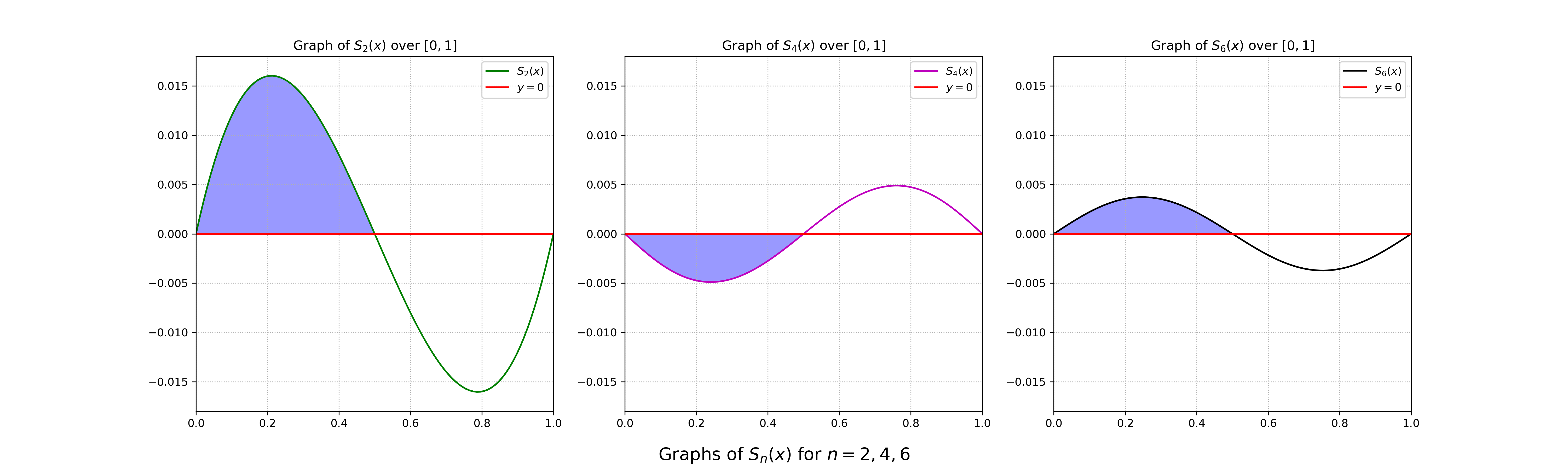}
Next, we compute certain generalized Bernoulli numbers $B_{n,\chi}$ where $\chi$ is a quadratic character of conductor $p$. Concretely, this is the character given by Legendre symbol
$\chi(a)=\left(\frac{a}{p} \right), a \in \ZZ.$ First, we provide a table of $B_{n,\chi}$ for $p \leq 11$. 
\begin{center}
\begin{tabular}{ |c||c|c|c|c|c|c|c|c|c|} 
 \hline
 $n$ & $B_n$ & $B_{n,\chi_{3}}$ & $B_{n, \chi_{4}}$ & $B_{n, \chi_{5}}$ &$B_{n,\chi_{7}}$  &$B_{n, \chi_{11}}$  \\
 \hline
  $0$ & 1      &0		 &0	      &	0	&0	&0							\\
 $1$ & -1/2  &-1/3 	&-1/2     &0	&-1	&-1									 \\ 
$2$ & 1/6    & 0		&0	      &	4/5	&0	&0									\\
$3$ & 0       &2/3	&3/2	      &	0	&48/7 &18	 									\\
$4$ & -1/30 &0		&0	      &	-8	&0	 &0	 							\\
$5$ & 0	   &-10/3	&-25/2   &0	&-160 &-12750/11 										 \\
$6$ &1/42.  &0		&0	       &804/5 &0	   &0		     							 \\ 
$7$ &0 	  &98/3	&427/2   &0		&8176 & 152082 								\\
$8$ &-1/30 &0		&0		&-5776	&0	    &	0.    							\\
$9$ &0	   &-1618/3 &-12465/2 &0		&-5086656/7 & -33743250 									 \\ 
$10$ & 5/66 &0		&0			&1651004/5 & 0   &	 0								\\ 
$11$ & 0 		&40634/3 &555731/2 &0		    &99070928 & 11392546506							\\ 
$12$ &-691/2730 &0  	& 0		&-27622104 &	0 	& 0							 \\ 
\hline
\end{tabular}
\end{center} 
\vspace{0.5cm}
Next, we provide a table of $B_{n,\chi}$ for $ 13 \leq p \leq 23$.

\begin{center}
\begin{tabular}{ |c||c|c|c|c|c|c|c|c|} 
 \hline
 $n$ & $B_{n,\chi_{13}}$ & $B_{n,\chi_{17}}$ & $B_{n, \chi_{19}}$ & $B_{n, \chi_{23}}$     \\
 \hline
  $0$ & 0      &0		 &0	      &	0								\\
 $1$ &0  &0 	&-1     &-3	   								 \\ 
$2$ & 4    & 8		&0	      &	0										\\
$3$ & 0       &0	&66	      &	144									\\
$4$ & -232 &-656		&0	      &	0	 	 							\\
$5$ & 0	   &0	&-13450   &-34080										 \\
$6$ &401556/13  &138984		&0	         &0     							 \\ 
$7$ &0 	  &0	&5303074   &18665136									\\
$8$ &-7482704	 &-958428704/17			&0		&0					\\
$9$ &0	   &0 &-66751985430/19	 &-17895000384		 									 \\ 
$10$ & 2890943420	 &37040430040			&0			&0    						\\ 
$11$ & 0 		&0 & 3539203405562		    &605747775717744/23	 							\\ 
$12$ &-1634752049016	 &-35766492971568	  	& 0		&0 
 						 \\ 
\hline
\end{tabular}
\end{center}

While working on these numerical data, we observed that in our examples, it is always the case that $v_p(B_{\frac{p-1}{2}, \chi})=-1$. Using the work of Ernvall (see \cite{[E]}), we are able to prove a more general statement.

\begin{rmk} 
After  finding a direct proof and searching further the literature, we also found Carlitz's result (see \cite[Theorem 3]{[Carlitz2]}) that implies the above observation. The announcement of results in \cite{[Carlitz2]} is in \cite{[Carlitz1]}. Observe however that there is a misprint in Theorem 1, line 4 in \cite{[Carlitz1]}. The relevant part of \cite[Theorem 3]{[Carlitz2]} states in our special case where values of character are just $1$ or $-1$ the following: ``Let $g$ be a primitive root mod $p$. If the conductor $f$ of the primitive character mod $f$ is a prime number $p$, $p>2$ ,then $\frac{B_n,\chi}{n}$ is integer unless $p$ and $1 - \chi( g)g^n$ are not coprime in which case ..." This statement does imply that $v_p(B_{\frac{p-1}{2},\chi}) =-1$ and also the more general version of our Proposition A.1. In this case (see \cite[Theorem 3]{[Carlitz2]} for the relevant notations) $\nu=0$ and using that $g$ is a primitive root mod $p$ we see that that $\gcd(p, 1-\chi(g)g^n)=\gcd(p, 1+g^n)>1$ iff $n \equiv \frac{p-1}{2} \pmod{(p-1)}$, and therefore we have that
$pB_{\frac{p-1}{2},\chi} \equiv -1 \pmod{p}.$
For the reader's convenience here we provide a direct short proof of this very interesting statement.
\end{rmk} 

\begin{prop}
Let $\chi$ be a quadratic character of conductor $p$ where $p$ is an odd prime number and $n \geq 0$. Then
\begin{enumerate}
\item If $n \not \equiv \frac{p-1}{2} \pmod{(p-1)}$ then $B_{n,\chi} \in \ZZ$.
\item If $n \equiv \frac{p-1}{2} \pmod{(p-1)}$ then $B_{n,\chi}=\frac{a}{p}$ where $a \in \ZZ$ and $a \equiv -1 \pmod{p}$. 
\end{enumerate}
\end{prop} 
\begin{proof}
By \cite[Theorem 1.2]{[E]}, we know that if $n \not \equiv \delta_{\chi} \pmod{2}$ then $B_{n, \chi}=0$. Furthermore, we remark that for all odd primes $p$
\[ \frac{p-1}{2} \equiv \delta_{\chi} \pmod{2}.\]
Therefore, it is is sufficient to prove this proposition when $n \equiv \delta_{\chi} \pmod{2}$.

By \cite[Theorem 1.4]{[E]}, we know that $pB_{n,\chi}$ is $p$-integral. Additionally, by \cite[Theorem 1.5]{[E]}, $B_{n,\chi}$ is $q$-integral for all primes $q \neq p$. Combining these two facts, we can conclude that $pB_{n, \chi} \in \ZZ$ for all $n \geq 0$. Furthermore, by \cite[Theorem 1.6]{[E]}, we know that 
\[ p^2 B_{n,\chi} \equiv \sum_{a=1}^{p^2} \chi(a) a^{n} \pmod{p^2} .\] 

We have the following congruences modulo $p^2$ 
\begin{align*}
\sum_{a=1}^{p^2} \chi(a) a^{n} &=\sum_{a=1}^{p} \left (\sum_{k=0}^{p-1} \chi(kp+a)(kp+a)^{n} \right) \\
                                                                   &=\sum_{a=1}^{p} \chi(a) \sum_{k=0}^{p-1} \left(a^{n}+n (kp) a^{n-1} \right)\\
                                                                   &=\sum_{a=1}^{p} \chi(a) \left(pa^{n}+np a^{n-1} \sum_{k=1}^{p-1} k \right) \\
                                                                   &=p \sum_{a=1}^{p} \chi(a)a^{n}+\frac{np^2(p-1)}{2} \sum_{a=1}^{p} \chi(a) a^{n-1} \\
                                                                   &=p \sum_{1}^{p} \chi(a) a^{n}\pmod{p^2}.
\end{align*} 
The second congruence comes from the expansion of $(a+kp)^{n}$ leaving out terms which are divisible by $p^2$. The last congruence comes from the identity
\[ \sum_{k=1}^{p-1} k=\frac{p(p-1)}{2} .\] 

We have the following well-known simple lemma.
\begin{lem}
Let $r$ be a natural number. Then 
\[ \sum_{a=1}^{p} a^{r} = \begin{cases} 0 &\mbox{if } p-1 \nmid r  \\
-1 & \mbox{if } p-1|r  \end{cases} .\] 
\end{lem} 
Using this lemma and the fact that $\chi(a) \equiv a^{\frac{p-1}{2}} \pmod{p}$ we have the following congruences modulo $p$ 
\[ \sum_{a=1}^{p} \chi(a) a^{n} \equiv \sum_{a=1}^{p} a^{n+\frac{p-1}{2}} \equiv \begin{cases} -1 &\mbox{if } n \equiv \frac{p-1}{2} \pmod{(p-1)}  \\
0 & \mbox{else. } \end{cases}\] 

Consequently, we have the following congruences modulo $p^2$ 
\[ p^2 B_{n,\chi} \equiv \begin{cases} -p &\mbox{if } n \equiv \frac{p-1}{2} \pmod{(p-1)}  \\
0 & \mbox{else. } \end{cases}\] 
Combining this congruence and the fact that $pB_{n,\chi} \in \ZZ$, the proposition follows easily. 
\end{proof}

\begin{acknowledgement}
Over the years after publication of the short note \cite{[Minac]}, Min\'a\v{c} received a number of encouraging correspondences and discussions concerning further generalizations and explorations of possible values of other zeta functions.  We are extremely grateful for our correspondents including but not limited to S. Chebolu, R. Dwilewicz, G. Everest, E. Frenkel, A. Granville, J. Merzel, L. Muller, \v{S}. Porubsk\'y, P. Ribenboim, Ch. Rottger, A. Schultz, M.Z Spivey, B. Sury. Further, we are are grateful to K. Dilcher and R. Ernvall for helping us to obtain some related references including the very nice R. Ernvall's thesis \cite{[E]}.
We are also grateful to an anonymous referee for his/her careful reading of the manusript, and for providing us with helpful comments and valuable suggestions. 
 Last but not least, we are grateful to Leslie Hallock for her  careful proofread. 
\end{acknowledgement}

\end{document}